\documentclass[12pt]{article}
\usepackage{latexsym}
\usepackage{amssymb,amsbsy,amsmath,amsfonts,amssymb,amscd,amsthm}
\usepackage{epsfig,graphicx,mathrsfs}
\usepackage{color}
\usepackage{subfigure}
\usepackage{pdfsync}
\usepackage{url}
\usepackage{algorithm}
\usepackage{algorithmic}

\newcommand{\dual}[2]{\langle#1,\,#2\rangle}
\newcommand{\norm}[1]{\left\Vert#1\right\Vert}
\newcommand{\R}{\mathbb{R}}

\newcommand{\N}{\mathbb{N}}

\newcommand{\E}{\mathbb{E}}

\def\b{\beta}
\def\d{\delta}

\def\g{\gamma}
\def\l{\lambda}
\def\L{\Lambda}
\def\m{\mu}
\def\n{\nu}

\def\s{\sigma}

\def\G{\Gamma}

\def\t{\tau}

\def\pd{\partial}

\newcommand{\mM}{\Gamma}
\def\mr{{\mu^R}}
\newcommand{\vv}{\tilde v}

\newcommand{\cB}{{\cal B}}
\newcommand{\cC}{{\cal C}}

\newcommand{\cE}{{\mathcal E}}

\newcommand{\cM}{{\cal M}}

\newcommand{\cP}{{\mathcal M}^+(\R^d)}

\newcommand{\fdz}{\partial_{(0,t]}^{\beta}}

\newcommand{\fIz}{I_{(0,t]}^{\beta}}

\newcommand{\Lip}{{\rm BL}(\R^d)}

\newcommand{\supp}{{\rm supp}}
\newcommand{\diver}{{\rm div}}
\newcommand{\normBL}[1]{\left\Vert#1\right\Vert^\ast_{BL}}
\newcommand{\wass}[2]{{d_{BL}}\left(#1,\, #2\right)}

\newtheorem{theorem}{Theorem}[section]
\newtheorem{lemma}[theorem]{Lemma}
\newtheorem{definition}[theorem]{Definition}
\newtheorem{proposition}[theorem]{Proposition}

\newtheorem{remark}[theorem]{Remark}

\numberwithin{equation}{section}
\begin{document}
\title{\Large \bf Memory effects in measure transport equations}
\author{Fabio Camilli\footnotemark[1] \and Raul De Maio\footnotemark[1]}
\date{\today}
\maketitle
\begin{abstract}
	Transport equations with a  nonlocal velocity field have been introduced as a continuum model for interacting particle systems   arising  in physics, chemistry  and biology. Fractional time derivatives, given by  convolution integrals of the time-derivative with power-law kernels,  are typical for memory effects in complex systems.
    In this paper we consider  a nonlinear transport equation  with  a fractional time-derivative. We provide     a well-posedness theory for weak measure solutions of the problem   and an integral formula which generalizes the classical push-forward  representation formula to this setting. 
\end{abstract}
 \begin{description}
    \item [{\bf AMS subject classification}:]35R11; 35Q35; 26A33. 
     \item[{\bf Keywords}:]  Transport equation; Measure-valued solution; Fractional derivative.
\end{description}
\footnotetext[1]{Dip. di Scienze di Base e Applicate per l'Ingegneria,  ``Sapienza'' Universit{\`a}  di Roma, via Scarpa 16,
00161 Roma, Italy, ({\tt e-mail: fabio.camilli,raul.demaio@sbai.uniroma1.it})}

\pagestyle{plain}
\pagenumbering{arabic}
\section{Introduction} 
 The measure-valued formulation of nonlinear transport equations  has received  an increasing interest both from a theoretical perspective \cite{ags,ccr} and in various applications,  as a 
 continuum model for interacting particle systems (e.g. in crowd motion, population dynamics, bacterial chemotaxis, kinetic theory, social systems, etc., see  \cite{cpt,ehm}
and reference therein).\par
Recently,   anomalous transport problems describing processes deviating from the usual  Gaussian behaviour  have been observed  in different fields (see  \cite{mk}). In these phenomena, the standard diffusive behaviour is replaced by a subdiffusive one, in which the 
mean square displacement of the diffusing particles is of order $t^\b$ with $\b<1$.  Corresponding  models
lead to the study of differential equations where the  time-derivative is replaced by a fractional one (see for example \cite{acv,ly,ob}). The stochastic dynamics driven by   fractional differential equations is usually referred to  as motion in a non-homogeneous medium. The particles can speed up or slow down according to a random clock which is the inverse to a stable subordinator (time-change of processes) and the dynamic is not Markovian, see  \cite{ms}.  Fractional derivatives     are also considered as a typical approach to add a   memory effect  to a  complex system. The memory term  introduced by the convolution operator in time finds applications in demography,  viscoelastic and  biomaterials, biological processes   and, in general, in  the study  of  constitutive relations depending on the history of the state variables
(see \cite{m,t}).\par
In this paper we consider a measure solution approach to the nonlinear transport equation
\begin{equation}\label{transport_intro}
\left\{\begin{array}{ll}
\fdz \m+\diver(v[\m_t]\,\m)=0\quad & (x,t)\in \R^d\times (0,T), \\[8pt]
 \m_{t=0}=\m_0  & x\in\R^d,
\end{array}
\right.
\end{equation}
where $\fdz\cdot$ is a nonlocal time-differential operator  given by the Caputo fractional derivative of order $\b\in (0,1)$
\[ \pd^{\b}_{(0,t]} \phi(t) =\frac{1}{\G(1-\b)}\int_0^t\frac{d\phi(\t)}{d\t}\frac{1}{(t-\t)^\b}d\t, \qquad t\in (0,T). \]
Since a fractional derivative  at time $t$ depends on the  values  of the time-derivative in the interval $(0,t)$ with a progressively decreasing weight on the past history, the problem is non local in time.
The coefficient  $v[\m](x)$ is a nonlinear  velocity field which depends   on the   solution itself, for example $v[\m](x)=\int_{\R^d} K(x-y)d\m(y)$. An appropriate   choice of  kernel $K$ allows to describe different phenomena of the physical model such as  aggregation, repulsion and diffusion (\cite{ccr,cpt,ehm}).  The  datum $\m_0$ is a given positive  measure on $\R^d$ representing the initial distribution of the population. Observe  that  measure solutions allow to
describe  within a unified  approach the evolution both for continuous and discrete  populations.\\
For $\b=1$, the fractional derivative $\fdz\cdot$ coincides with the standard  derivative $\pd_t \cdot$. In this case the measure solution of   problem \eqref{transport_intro}   is given by the push-forward $\m_t=\Phi^\mu_t\# \m_0$   of the initial datum $\m_0$ 
by means of the flow map $\Phi^\mu_t$ associated to the velocity field $v[\mu]$.  
In this paper,  we extend the previous results to the fractional case, providing     a well-posedness theory for  measure solutions of the problem  \eqref{transport_intro}  and an integral formula which generalizes the classical push-forward  representation formula to this setting. Indeed, we prove that 
for $\b\in(0,1)$  a measure solution of \eqref{transport_intro} is given by the integral formula 
\begin{equation}\label{push_intro}
\m_t(dx)= \int_0^\infty m^\m_s(dx)h_\b(s,t)ds=\int_0^\infty \Phi^{\m}_s\# \m_0(dx)\,h_\b(s,t)ds
\end{equation}
where  $m^\m$ is the solution of the linear transport equation
\begin{equation*}
\left\{\begin{array}{ll}
\pd_r m_s+\diver(\vv^\m(x,s) m_s)=0\quad & (x,s)\in \R^d\times\R^+, \\[8pt]
m_{s=0}=\m_0  & x\in\R^d,
\end{array}
\right.
\end{equation*}
with the velocity field given by $\vv^\m(x,s)=\int_0^\infty v[\m_r]g_\b(r,s)dr$ and $\Phi^{\m}$ the associated flow.
The kernel $g_\b(s,t)$, $h_\b(s,t)$ are the  probability density functions (PDF in short) associated to  the time-increasing, continuous stochastic processes $D_t$, $E_t$, which are a $\b$-stable subordinator  and its inverse (\cite{hku,mk,ms}). Note that, as in the classical case,  \eqref{push_intro} defines implicitly the solution of \eqref{transport_intro} since the velocity field $\vv^\m(x,s)$ depends on the solution $\mu$ of the problem.
\\
The paper is organized as follows. In Section \ref{sec:prelim}, we review  definitions and some basic properties of the subordinator   process, of  its inverse and of the associated  PDF  $g_\b$, $h_\b$. In Section \ref{sec:linear}, we study a linear transport equation with Caputo time-derivative. In  
Section \ref{sec:nonlinear} we prove the well-posedness of weak measure solutions to  problem \eqref{transport_intro}.
\section{Preliminary definitions and properties}\label{sec:prelim}
Throughout the paper, we always assume that $\b\in (0,1]$. For   $f:\R^+ \to \R$, the  Riemann-Liouville fractional integral of order $\b$ is defined by
\[
I^{\b}_{(0,t]} f(t):=\frac{1}{\G(\b)}\int_0^t f(\t)\frac{1}{(t-\t)^{1-\b}}d\t,\qquad t\in\R^+, 
\]
and the Caputo  fractional    derivative  of order $\b$  by
\[
\pd^{\b}_{(0,t]} f(t):=I^{1-\b}_{(0,t]} \left[\frac{df}{dt}(t)\right] =\frac{1}{\G(1-\b)}\int_0^t\frac{df}{dt}(\t)\frac{1}{(t-\t)^\b}d\t,\qquad t\in\R^+, 
\]
(for a   complete account of the theory of fractional derivatives, we refer to \cite{po}).\\
Fractional derivatives appear in the study of differential equations which govern the evolution of probability density functions for a class of L\'evy  processes, called stable subordinators  (\cite{hku,ms,mst}).  For $\b\in(0,1)$, a $\b$-stable subordinator is a one-dimensional, non-decreasing L\'evy process $D_t$ starting at $0$ which is self-similar , i.e. $\{D_{t},t\ge 0 \}$ has the same  finite dimensional distribution as $\{t^{1/\b}D_1,t\ge 0 \}$, and such that  the Laplace transform of $D_1$ is given by $\E(e^{-sD_1})= e^{-s^{\beta}}$  for $s\ge 0$. The inverse stable  process $\{E_t\}_{t\geq 0}$, defined as the first passage time of the process $D_t$ over the level $t$, i.e.
$$E_t = \inf\{\t>0 : D_\t > t\},$$
has sample paths which are continuous,  non-decreasing and such that $E_0 = 0$,  $E_t \to \infty$ as $t \to \infty$. 
The PDF  of the process $E_t$ is given by $h_\beta(s,t)=\frac{t}{\beta}s^{-1 - \frac{1}{\beta}}G_\b(s^{- \frac{1}{\beta}}t)$
where $G_\b(\cdot)$ is PDF  of   $D_1$.  The  function $h_\b(\cdot,t)$ is infinitely differentiable in $\R^+$,  right continuous in $0$  with 
\begin{equation}\label{eq_too}
h_\b(0^+,t)=\frac{1}{\Gamma(1-\b)} t^{-\b} \qquad \text{ for $t>0$},
\end{equation} 
and has finite $\g$-moment for any $\g>0$, given by (see \cite[Corollary 3.1]{ms} and \cite[equation (2.7)]{pi})
\begin{equation}\label{moment}
\E[E_t^\g] =\int_0^\infty s^\g\,  h_\b(s,t)ds= C(\beta,\gamma)t^{\g\beta}.
\end{equation}
   with  $C(\beta, \g)= \frac{\Gamma(\gamma + 1)}{\Gamma(\gamma \beta + 1)}$. 
In particular, this property   implies  the   identity 
\begin{equation}\label{exp}
\mathbb{E}(e^{\lambda E_t}) = \cE_{\beta}(\lambda t^{\beta}), \quad \forall \lambda \in \R,
\end{equation}
where $\cE_\beta(z)=\sum_{k=0}^{+\infty}\frac{z^k}{\G(\beta k+1)}$ is the Mittag-Leffler function of order $\beta$ (see \cite{po} for some properties of $\cE_\beta$).  
  Indeed, we have
\begin{align*}
\mathbb{E}(e^{\lambda E_t}) &= \int_0^{+\infty}e^{\lambda s}h_\beta(s,t) ds =  \sum_{k=0}^{+\infty} \lambda^k \int_0^{+\infty} \frac{s^k}{k!} h_\beta(s,t)ds \\&= \sum_{k=0}^{+\infty}\frac{\lambda^k}{k!} \frac{\Gamma(k+1)}{\Gamma(k\beta + 1)}t^{\beta k}  = \sum_{k=0}^{+\infty} \frac{(\lambda t^\beta)^k}{\Gamma(k\beta + 1)} = \cE_\beta(\lambda t^\beta).
\end{align*}
In the next proposition, we give two crucial properties of $h_\b$ we will exploit in the following  (see \cite{mst} for {\it (i)}  and \cite[Lemma 3.2]{hku} for {\it (ii)}).
\begin{proposition}\label{prp1}
	\begin{itemize}
		\item[(i)] For $t> 0$, $h_\b$ is a weak solution of
		\begin{equation}\label{eq_E}
		\fdz h_\b(r,t)=-\pd_r h_\b(r,t)-\frac{t^{-\b}}{\Gamma(1-\b)} \d_0(r), \qquad r\in \R^+ .
		\end{equation}
		\item[(ii)] Given $t>0$, the density $h_\b(\cdot,t)$ is bounded and there exist positive constants $C$, $K$, non depending on $r$, such that
		\begin{equation}\label{decay}
		  h_\b(r,t)\le \frac{C}{\b} t^{-\frac{\b}{2}\frac{1}{1-\b}}r^{-1 + \frac{1}{2(1-\beta)}} e^{-Kt^{\frac{\b}{\b-1}}r^{\frac{1} {1-\b}}}
		\end{equation}
		for $r$ large enough.
	\end{itemize}
\end{proposition}
 \begin{lemma}\label{lemm1}
Given $\lambda > 0$, the function   $\Psi(t) := \mathbb{E}[\lambda E_t e^{\lambda E_t}]$, where   $ E_t$ is the inverse stable subordinator, is well-defined for every $t\geq 0.$
\end{lemma}
\begin{proof}
Since $ E_{t} \equiv 0$ for $t=0$, we immediately get that $\Psi(0) = 0.$ For $t > 0$, we have by definition
\begin{equation}
\Psi(t) = \int_0^{ \infty} \lambda s e^{\lambda s} h_{\beta}(s,t) ds.
\end{equation}
Hence, by equation \eqref{eq_E}, we get
\begin{align*}
\fdz \Psi (t) &= \int_0^{ \infty }\l s e^{\l s} \fdz h_\b (s,t) ds \\
&= - \int_0^{+\infty} \l s e^{\l s} \partial_s h_\b(s,t) ds - \frac{t^{-\b}}{\G(1-\b)} \langle \delta_0, \l s e^{\l s}\rangle\\
&= - \lambda s e^{\l s} h_\beta (s,t) |_{0}^{+\infty} + \int_{0}^{+\infty} (\l e^{\l s} + \l^2 s e^{\l s}) h_\b(s,t) ds.
\end{align*}
Fixed   $t>0$, by \eqref{eq_too}  the term $\lambda s e^{\l s} h_\beta (s,t)$ converges to 0 for $s \to 0^+$. Moreover, by   inequality \eqref{decay}, we have for $s \to +\infty$ 
\begin{align*}
0 \leq \l s e^{\l s} h_\beta (s,t) \leq \l \frac{C}{\beta} t^{-\frac{\b}{2}\frac{1}{1-\b}}s^{  \frac{1}{2(1-\beta)}} e^{\l s -Kt^{\frac{\b}{\b-1}}s^{\frac{1} {1-\b}}}.
\end{align*} 
Since $\b \in (0,1)$, it follows that $\frac{1}{(1-\b)}>1$ and therefore $\lambda s e^{\l s} h_\beta (s,t)$ converges to 0 also for $s\to +\infty$. Hence, by \eqref{exp},  we obtain 
\begin{equation}\label{fode}
\fdz \Psi (t) = \l \cE_\beta(\l t^\b) + \l \Psi(t).
\end{equation}
By the theory of linear fractional differential equations (see \cite[Theorem 3.3]{po}), 
the initial value problem
\[
\left\{ 
\begin{array}{ll}
\fdz \psi (t) = \l \cE_\beta(\l t^\b) + \l \psi(t),\qquad t\in (0,T)\\
\psi(0)=0,
\end{array}
\right.\]
admits a unique continuous solution in $[0,T]$, for any $T>0$. Since $\Psi: \R^+ \to \R^+$ solves the initial value  problem, then it is well defined for any $t\ge 0$.
\end{proof}
%
We introduce a space of measures with an appropriate norm where we consider    solutions to the measure-valued transport equation (we  refer to \cite{ags,ehm} for a comprehensive account of the theory). 
We denote by $\cM(\R^d)$ the space of finite Borel measures on $\R^d$ and by $\cM^{+}(\R^d)$ the convex cone of the positive measures in $\cM(\R^d)$. For $\mu\in\cM(\R^d)$ and a bounded measurable function $\varphi:\R^d\to\R$ we write
$$ \dual{\mu}{\varphi}:=\int_{\R^d}\varphi\,d\mu. $$\\
Given a Borel measurable vector field $\Phi:\R^d\to\R^d$,  we denote by $\Phi\#\mu\in\cM(\R^d)$ the  push-forward of the measure $\mu$ under the action of $\Phi$, defined by
$$ (\Phi\#\mu)(E):=\mu(\Phi^{-1}(E)), \qquad \forall\,E\in\cB(\R^d), $$
where $\cB(\R^d)$ is the class of Borel sets in $\R^d$.
Observe that $\dual{\Phi\#\mu}{\varphi}=\dual{\mu}{\varphi\circ\Phi}$.
We denote by $BL(\R^d)$ the Banach space of bounded and Lipschitz continuous functions $\varphi:\R^d\to\R$
 equipped with the norm
$$\|\varphi\|_{BL} = \|\varphi\|_{\infty} + |\varphi|_{L},$$
where $|\varphi|_L$ is the Lipschitz seminorm, and  we introduce a norm in $\cM(\R^d)$ by taking the dual norm of $\norm{\cdot}_{BL}$:
$$ \|\mu\|_{BL}^*:=\sup_{\substack{\phi\in BL(\R^d) \\ \|\phi\|_{BL} \leq 1}}\dual{\mu}{\phi}. $$
 Indeed, given the norm $\|\cdot\|_{BL}$ on $BL(\R^d)$, the dual norm $\|\cdot\|_{BL}^*$ is, trivially, defined on the dual space $BL(\R^d)^*$. Since the map $I_\mu(\phi) := \langle \mu, \phi\rangle$ defines a linear embedding of $\cM(\R^d)$ into $BL(\R^d)^*$, we can induce a norm, denoted by the same symbol, on $\cM(\R^d)$.
The previous norm is said {\em Bounded Lipschitz norm}   on $\cM(\R^d)$. \\
The space $(\cM(\R^d),\,\|\cdot\|_{BL}^*)$ is in general not complete, hence it is customary to consider its completion $\overline{\cM(\R^d)}^{\normBL{\cdot}}$ with respect to the dual norm. However, the cone $\cM^+(\R^d)$, which is a closed subset of the completion of  $\cM(\R^d)$ in the weak topology, is complete, although it is  not  a Banach space because it is not a vector space. Since we consider only positive measures,   we restrict our attention to the complete metric space $(\cM^+(\R^d),\,d_{BL})$ with the    distance $d_{BL}$ induced by the norm $\|\cdot\|_{BL}^*.$
\begin{remark}
  In	\cite{ags,cpt},   solutions to transport equations are considered in the space of probability measures $\mathcal{P}_1(\R^d)$ with the norm induced by  the  Wasserstein distance $W_1$. Note that, by the Kantorovich's dual representation theorem, it follows that 
\begin{equation}\label{equiv_Kant}
\|\mu - \eta\|_{BL}^* \leq W_1(\mu, \eta),\quad \forall \mu, \eta \in \mathcal{P}_1(\R^d).
\end{equation}
Since  we consider transport equations with a source term in Remark \ref{source} and, in this case, conservation of the mass in general does not hold, we prefer to opt for the dual Bounded Lipschitz norm $\|\cdot\|_{BL}^*$ to have an unified framework.\\
	The distance induced in $\cM(\R^d)$ by the \emph{total variation norm}:
	$$ \norm{\mu}_{TV}:=\sup_{\substack{\varphi\in C_b(\R^d) \\ \norm{\varphi}_\infty\leq 1}}\dual{\mu}{\varphi}, $$
	where $C_b(\R^d)$ is the space of bounded continuous function on $\R^d$, is another metric frequently used for measures. However, as observed in \cite{cpt},  it may not be fully suited to transport problems where one wants to measure the distance between flowing mass distributions.
\end{remark}

\section{Linear transport equations with Caputo time derivative}\label{sec:linear}
In order to explain the construction of a solution to \eqref{transport_intro} in a simpler setting, in this section we  consider the case of a linear transport equation with a Caputo time-derivative
\begin{equation}\label{caputo_linear}
\left\{\begin{array}{ll}
\fdz \m+\diver(v(x,t)\,\m)=0\quad & (x,t)\in \R^d\times (0,T), \\[8pt]
\m_{t=0}=\m_0  & x\in\R^d,
\end{array}
\right.
\end{equation}
where $\m_0\in\cP$ and $T > 0$, and we prove that the previous problem is well posed in measure theoretic sense. We start by  introducing a notion of measure solution to \eqref{caputo_linear}.
\begin{definition}\label{def_weak_linear}
	A solution to \eqref{caputo_linear} is a map $\m\in C( [0,T],\cP)$ such that $\m_{t=0}=\m_0$ and for any $f\in C^\infty_c(\R^d)$ and for almost any $t\in [0,T]$,
	\[
	\fdz\int_{\R^d} f(x) \m_t(dx)=\int_{\R^d}Df(x)\cdot v(x,t) \,  \m_t(d x).
	\] 
	Equivalently, since $\fIz[\fdz \phi](t)=\phi(t)-\phi(0)$ for an absolute continuous function $\phi$, $\m_t$ is a solution to \eqref{caputo_linear} if
	\[	  
	\int_{\R^d} f(x) \m_t(dx)= \int_{\R^d} f(x) \m_0(dx)
	+ \fIz\left[ \int_{\R^d}Df(x)\cdot v(x,t)\, \m_t(d x)\right] .
	\]
\end{definition}
We assume that $v:\R^d\times\R^+\to\R^d$ is a given vector field satisfying
\begin{itemize}
	\item [\textbf{(H1)}] $v$ is  bounded by  $V_0>0$, measurable in $t$  and    there exists  $L\in \R^+$	such that,  for any $x_1,x_2 \in \R^d$ and $t\in\R^+$, it holds 	 
	$$|v(x_1,t) - v(x_2,t)|\leq L   |x_1 - x_2|.$$
\end{itemize}
Associated to \eqref{caputo_linear}, we consider the linear problem with standard time-derivative
\begin{equation}\label{dual_linear}
\left\{\begin{array}{ll}
\pd_t m+\diver(\vv(x,t)\,m)=0\quad & (x,t)\in \R^d\times \R^+, \\[8pt]
m_{t=0}=\m_0  & x\in\R^d.
\end{array}
\right.
\end{equation}
where the vector field $\vv:\R^d\times\R^+\to\R^d$ is defined by
\begin{equation}\label{dual_vel_linear}
\vv(x,t):=\E[v(x,D_t)]=\int_0^\infty v(x,s)g_\b(s,t)ds  
\end{equation}
(the process $D_t$ is the $\b$-stable subordinator  defined in Section \ref{sec:prelim} and $g_\b$ the corresponding PDF). The vector field $\tilde v$ is continuous in $t$ since
\begin{equation}\label{pdf}
g_\b(s,t)= \frac{1}{t^{1/\b}}G_\b(\frac{s}{t^{1/\b}})
\end{equation} 
and $G_\b$   is infinitely differentiable in $\R^+$ (see \cite{ms}).
For any $x_1,x_2 \in \R^d$ and  $t\in\R^+$, we have 
\begin{equation}\label{Ltilde}
 |\vv(x_1,t) - \vv(x_2,t)| \le \int_0^\infty |v(x_1,s)-v(x_2,s)|g_\b(s,t)ds \leq L  |x_1 - x_2|.
\end{equation} 
Hence  the corresponding flow 
\begin{equation*}
\Phi_t(x,0):=x+\int_0^t \vv(\Phi_s(x,0),s)\,ds, 
\end{equation*}
giving  the trajectory issuing from the point $x$ at time $0$ and arriving at the point $\Phi_t(x,\,0)$ at time $t$, is well defined.\\
It is well known that the unique measure solution to \eqref{dual_linear}  is defined by the push-forward $m_t=\Phi_t\# \m_0$ of the initial datum $\m_0$, i.e.  $\dual{m_t}{f}=\int_{\R^d} f(\Phi_t(x,0))\m_0(dx)$ for any $f\in C^\infty_c(\R^d)$ (see for example  \cite{cpt,pr}).  Because problem \eqref{dual_linear} is well posed in measure theoretic sense, we get the corresponding result for problem \eqref{caputo_linear}.
\begin{proposition}\label{existence_linear}
	Assume \textbf{(H1)}. 
For any $T>0$, the Cauchy	problem  \eqref{caputo_linear}  admits a  solution $\m\in C([0,T],\cP)$ given by 
$\m_t(dx)= \E[m_{E_t}(dx)]$, i.e.
\begin{equation}\label{integral_l}
\m_t(dx)= \int_0^\infty m_s(dx)h_\b(s,t)ds=\int_0^\infty \Phi_s\#\m_0(dx) \, h_\b(s,t)ds,
\end{equation}
where  $m\in C(\R^+,\cP)$	is the  solution of the linear transport problem \eqref{dual_linear}.  
Moreover, 
let
$\m^1,\m^2$ be two solutions of \eqref{caputo_linear} corresponding to the initial data $\m^1_0, \m^2_0$ . Then, there is a constant $C=C(T)$ such that
\begin{equation}\label{uniqueness_linear}
\sup_{[0,T]}\wass{\m^2_t}{\m^1_t} \leq  C \wass{\m^2_0}{\m^1_0}.
\end{equation}
\end{proposition}
\begin{proof}
	By assumption \textbf{(H1)}, the flow $\Phi_s$ exists for any $s\in\R^+$ and the push-forward $m_s=\Phi_s\# \m_0$ is globally well defined. Therefore also    formula \eqref{integral_l} is well defined for any $t\in [0,T]$. Moreover $\m_t$, for any $t\in\R^+$, is a finite measure on $\R^d$ since we have
	\begin{equation*} 
	\m_t(\R^d)=\int_{\R^d}\m_t(dx)=\int_0^\infty \int_{\R^d}   m_{r}(dx) h_\b(r,t)dr=\int_0^\infty \int_{\R^d}   \m_0(dx)  h_\b(r,t)dr=\m_0(\R^d).
	\end{equation*}
We claim  that \eqref{integral_l} defines a weak solution to \eqref{caputo_linear}.  Since $h_\b$ satisfies \eqref{eq_E}, \eqref{decay},  we have
	\begin{align*}
	&\fdz\left( \int_{\R^d}f(x)\,\m_t(dx)\right)=\fdz \left( \int_{0}^{\infty}  \int_{\R^d}f(x)\,m_r(dx)h_\b(r,t)\,dr\right) \\
	&= \int_{0}^{\infty}\left( \int_{\R^d}f(x)\,m_r(dx)\right)\fdz h_\b(r,t)\,dr=-\int_{0}^{\infty}\left( \int_{\R^d}f(x)\,m_r(dx)\right)\pd_r h_\b(r,t)\,dr\\
	&
 - \langle \left( \int_{\R^d}f(x)\,m_r(dx)\right) \frac{t^{-\b}}{\Gamma(1-\b)}, \d_0(r)\rangle
	=\int_{0}^{\infty} \frac{d}{dr}\left(\int_{\R^d}f(x)\,m_r(dx)\right) h_\b(r,t)\,dr\\
	& -\left[ h_\b(r,t)\int_{\R^d}f(x)\,m_r(dx)\right]_{r=0}^{r=\infty}-\frac{t^{-\b}}{\Gamma(1-\b)}\int_{\R^d}f(x)\,m_0(dx)\\
	& =\int_{0}^{\infty} \frac{d}{dr}\left(\int_{\R^d}f(x)\,m_r(dx)\right) h_\b(r,t)\,dr.
 	\end{align*}
	Moreover, 	since $D_{E_t}=t$ for any $t$ and recalling that $h_\b(\cdot,t)$, $g_\b(\cdot,t)$ are the PDFs of  $E_t$, $D_t$, respectively, we have 
	\begin{align*}
	&\int_{\R^d}Df(x)\cdot v(x,t)\,  \m_t(d x)=\E\left[\int_{\R^d}Df(x)\cdot v(x,{D_{E_t}})\,  m_{E_t}(d x)\right]\\ &=\E\left[\int_0^\infty\int_{\R^d}Df(x)\cdot v(x,{D_r})\,  m_r(d x)h_\b(r,t)dr\right]\\
	&	=	\int_0^\infty\int_{\R^d}Df(x)\cdot \E\left[v(x,{D_r})\right]m_r(dx) h_\b(r,t)dr =	\int_0^\infty\left[ \int_{\R^d}Df(x)\cdot \vv(x,r)\right]m_r(dx) h_\b(r,t)dr.
	\end{align*}
	Replacing the previous  identities in 
	Definition \ref{def_weak_linear}, we get that $\m$ is a solution if
	\[  
	\int_{0}^{\infty}\left( \frac{d}{dr}\int_{\R^d}f(x)\,m_r(dx) -\int_{\R^d} Df(x)\cdot  \vv(x,r)\,m_r(dx)\right) h_\b(r,t)\,dr=0.
	\]
	Recalling that $m$ is a weak solution to \eqref{dual_linear}, i.e.
	for any $f\in C^\infty_c(\R^d)$ and for almost any $t\in [0,T]$
	\[
	\frac{d}{dr}\int_{\R^d} f(x)  m_r(dx)=\int_{\R^d}Df(x)\cdot \vv(x,r)\,  m_r(d x),
	\] 
	we get the claim.\\
To prove that $\m$ is continuous with respect to $t$, we estimate $\wass{\m_t}{\m_{t'}}$ for $0\le t'<t$. For $f\in \Lip$ such that $\|f\|_{BL}\le 1$ we have
\begin{equation}\label{m2_l}
\begin{split}
 \dual{\m_t-\m_{t'}}{f}&=\int_0^\infty  \left( \dual{m_{r}}{f}h_\b(r,t)-\dual{m_{r}}{f}h_\b(r,t')\right)dr=\int_0^\infty   \dual{m_{r}}{f}\left(h_\b(r,t)-h_\b(r,t')\right)dr.
 \end{split}
\end{equation}
By  the estimate for $r,r'\in\R^+$,
\begin{equation}\label{m3_l}
\begin{split}
&\dual{m_r-m_{r'}}{f}=\int_{\R^d}(f(\Phi_r(x,0))-f(\Phi_{r'}(x,0))\m_0(dx)
\\&\le \int_{\R^d}\left|\int_0^r\vv(\Phi_s(x,0),s)-\int_0^{r'}\vv(\Phi_s(x,0),s)\right|\m_0(dx) \le V_0\m_0(\R^d)|r-r'|,
\end{split}
\end{equation}
we get for $0 \leq t' < t$
\begin{equation}\label{m3_I2}
\begin{split}
&\int_0^\infty   \dual{m_{r}}{f}\left(h_\b(r,t)-h_\b(r,t')\right)dr  =
\E\left[\dual{m_{E_t} -m_{E_{t'}}}{f} \right]\\ &\le V_0\m_0(\R^d) \E\left[ E_t-E_{t'}\right] =V_0\m_0(\R^d)C(\b,1) (t^\b-t'^\b)\le V_0\m_0(\R^d)C(\b,1) |t-t'|^\b.
\end{split}
\end{equation}
Replacing the previous estimates in \eqref{m2_l}, we obtain
for the arbitrariness of $f$
\[
\wass{\m_t}{\m_{t'}}\le   C(\beta,1)V_0\m_0(\R^d)  |t-t'|^\b.
\]
We finally prove \eqref{uniqueness_linear}. 
Let $f\in \Lip$ be such that $\|f\|_{BL}\le 1$. By Gronwall's inequality,  
the function $L^{-1}f(\Phi_s(x,0))e^{-Ls}$ is $1$-Lipschitz. 
Then
\begin{align*}
\dual{\m^2_t-\m^1_t}{f} &= \int_0^\infty \dual{m^2_s-m^1_s}{f}  h_\b(s,t)ds\\
&=  \int_0^\infty \int_{\R^d}f(\Phi_s(x,0)) (\m_0^2(dx)-\m_0^1(dx)) h_\b(s,t)ds\\
&\le \wass{\m_0^2}{\m_0^1}\int_0^\infty Le^{Ls} h_\b(s,t)ds\\
&\le 
  \wass{\m_0^2}{\m_0^1} \int_0^\infty L e^{Ls}h_\b(s,t) = L\cE_\b(L t^{\b})\wass{\m_0^2}{\m_0^1}  
\end{align*}
where we have used \eqref{exp} in the last equality.
\end{proof}
\begin{remark}
If $\m_0=\d_{x_0}$, then the solution of \eqref{dual_linear}  is given by $\d_{\Phi_{t}(x_0,0)}$, while  the solution of \eqref{caputo_linear}   by 	$\m_t=\E\left[ \d_{\Phi_{E_t}(x_0,0)}\right]$.
We can interpret this formula in the following way: for the single particle the standard time $t$ is replaced by an internal clock $E_t$. The sample of the process $E_t$ can  be constant on some interval, corresponding to trapping events in the motion, and  assume arbitrarily large values, but
with a probability decaying exponentially to 0 (see \eqref{decay}). The    solution of the transport equation \eqref{caputo_linear} is obtained by averaging with respect to the internal clock  $E_t$. 
\end{remark}

\begin{remark}
Assume that the velocity $v$ is a given vector field in $L^1(\R^+;L^\infty(\R^d))^d$  satisfying the One-Sided Lipschitz (OSL) condition
\[  (v(x,t)-v(y,t))\cdot (x-y)\le \alpha(t)|x-y|\] 
for $\alpha\in  L^\infty(\R^+)$. 
Because of the weak regularity of $v$, it is natural to consider the characteristic flow $\Phi$ associated to $v$ in the sense of Filippov.
In \cite{pr} it is proved  that the push-forward of the initial datum by means of the Filippov flow $\Phi$ gives a unique measure solution  to the  corresponding transport problem with a  standard time-derivative. \\
Since the velocity $\vv$ defined in \eqref{dual_vel_linear} satisfies the same assumptions as $v$, including the OSL condition,
the solution to problem \eqref{dual_linear} is well defined for any $s\in\R^+$. Therefore, also in this weaker
setting, formula \eqref{integral_l} gives the solution to \eqref{caputo_linear}.
\end{remark}
\section{Nonlinear transport equations with Caputo time derivative}\label{sec:nonlinear}
In this section  we consider the Cauchy problem for a nonlinear transport equation with Caputo time-derivative, i.e.
\begin{equation}\label{caputo_nonlinear}
\left\{\begin{array}{ll}
\fdz \m_t+\diver(v[\m_t]\,\m_t)=0\quad & (x,t)\in \R^d\times   \R^+, \\[8pt]
\m_{t=0} =\m_0  & x\in\R^d.
\end{array}
\right.
\end{equation}
Also in this case, we consider solutions in the measure theoretic sense. 
\begin{definition}\label{def_weak}
	A solution to \eqref{caputo_nonlinear} is a map $\m\in C(   \R^+,\cP)$ such that $\m_{t=0}=\m_0$ and for any $f\in C^\infty_c(\R^d)$ and for almost any $t\in  \R^+$
	\[
	\fdz\int_{\R^d} f(x) \m_t(dx)=\int_{\R^d}Df(x)\cdot v[\m_t]\,  \m_t(d x).
	\] 
\end{definition}
Along this section we assume that
\begin{itemize}
	\item[\textbf{(H2)}] $v$ is  bounded by  $V_0>0$ and Lipschitz continuous, i.e. there exists $L > 0$ such that  for any $x_1,x_2 \in \R^d$, $\mu_1, \m_2 \in \cP$,	 
		\[  |v[\mu_1](x_1) - v[\mu_2](x_2)|\leq L \left( |x_1 - x_2|+ \wass{\m_1}{\mu_2} \right). \]
	\item[\textbf{(H3)}]  $\m_0\in\cM^+(\R^d)$ and, for  given constants $C_{k}$, 
	$$ \int_{\R^d}|x|^k\m_0(dx)\le C_{k},\qquad\text{for $k=1,2$.}
	$$ 
	\end{itemize}
The previous assumptions are standard in the framework of the nonlinear transport   theory (see \cite{cpt}).
For fixed $\m\in C( \R^+,\cP)$, we consider the linear problem  
\begin{equation}\label{dual_nonlinear}
\left\{\begin{array}{ll}
\pd_t m+\diver(\vv^\m(x,t)\,m)=0\quad & (x,t)\in \R^d\times \R^+, \\[8pt]
m_{t=0}=\m_0  & x\in\R^d,
\end{array}
\right.
\end{equation}
where the linear velocity field $\vv^\m:\R^d\times\R^+\to\R^d$ is defined by
\begin{equation}\label{dual_vel_nonlinear}
\vv^\m(x,t):=\E\big[v[\m_{D_t}](x)\big]=\int_0^\infty v[\m_s](x)g_\b(s,t)ds.
\end{equation}
Observe that  $\vv^\m$ is   bounded by $V_0$. Moreover, we have
\begin{align*}
|\vv^\m(x_1,t) - \vv^\m(x_2,t)|&\le \int_0^\infty \left|v[\m_s](x_1)-v[\m_s](x_2)\right|g_\b(s,t)ds\\
&\leq L|x_1 - x_2|\int_0^\infty  g_\b(s,t)ds=L|x_1-x_2|,
\end{align*}
and therefore $\vv^\m$ is Lipschitz continuous in $x$. 
Moreover, by \eqref{pdf} and  since the function $G_\b$ is infinite differentiable in $\R^+$,   $\vv^\m$ is continuous in $t$. 
Hence, for fixed $\m$,     the   flow 
 \begin{equation}\label{flow_nl}
 \Phi^\m_t(x,0):=x+\int_0^t \vv^\m(\Phi^\m_s(x,0),s)\,ds, 
 \end{equation}
 is well defined and the measure solution to the linear problem \eqref{dual_nonlinear} is given 
 by $m^\m_s=\Phi^\m_s\#\m_0$.\\
In the next theorem, we prove  the  well-posedness of problem \eqref{caputo_nonlinear}, showing existence and uniqueness of the solution  in  weak measure  sense,   and  we provide an integral formula which generalizes the classical push-forward  representation formula to this setting. 
\begin{theorem}\label{existence_nl}
Assume \textbf{(H2)}-\textbf{(H3)}. Then,
 the Cauchy	problem  \eqref{caputo_nonlinear}  admits \textbf{a unique}  solution $\m\in C( \R^+,\cP)$,  given by the formula
$\m_t(dx)=\E[m^\m_{E_t}]$  or, equivalently, by
\begin{equation}\label{integral_nl}
\m_t(dx)= \int_0^\infty m^\m_s(dx)h_\b(s,t)ds=\int_0^\infty \Phi^\m_s\#\m_0(dx) \, h_\b(s,t)ds
\end{equation}
where  $m^\m\in C(\R^+,\cP)$	is the  solution of the linear transport problem \eqref{dual_nonlinear}.
\end{theorem}
In the next lemma, we prove existence of a solution to \eqref{caputo_nonlinear}  when the initial datum $\m_0$ has compact support and the velocity is null outside a given ball.
\begin{lemma}\label{existence_compact}
	Besides \textbf{(H2)}-\textbf{(H3)}, assume that $\supp\{\m_0\}\subset B(0,R)$ and $v[\m](x)\equiv 0$ for all $\m\in\cP$, $x\in\R^d\setminus B(0,2R) $, for some positive constant
	$R$. Then the Cauchy	problem  \eqref{caputo_nonlinear}  admits a  solution $\m\in C( \R^+,\cP)$ implicitly defined  by the  integral formula \eqref{integral_nl}.
\end{lemma}
\begin{proof}
 Let   $\{T_n\}_{n \in\N}$ be an increasing sequence such that $\lim_{n\to \infty} T_n=+\infty$ and, for any $n\in\N$, define the map $\Psi_n: C( \R^+ ,\cP) \to C( \R^+,\cP)$ which  associates to   $\nu \in  C( \R^+,\cP)$  a measure-valued curve  $\mu \in  C(\R^+,\cP)$, defined for every $t \in \R^+$ as 
		$$\mu_t(dx)=\int_0^{T_n} m^{\n}_r(dx)h_\b(r,t)dr,$$
where $m^{\n}$  is the solution of the linear transport equation
		\[ 
		\left\{\begin{array}{ll}
		\pd_r  m_r+\diver(\vv^\n(x,r)\, m_r)=0\quad & (x,r)\in \R^d\times \R^+, \\[8pt]
		m_{r=0} =\m_0                           & x\in\R^d,
		\end{array}
		\right. \]
	and 	$\vv^\n$ is defined as in \eqref{dual_vel_nonlinear} with $\n$ in place of $\m$
		and $\Phi^\n$ is the associated flow.\\
Let us check that, for any $n\in\N$, $\Psi_n$	 is well-defined and continuous from the set
\begin{align*}
		\cC :=\left\lbrace  \m\in C(\R^+,\cP):\,\sup_{t\neq t'}\frac{\wass{\m_t}{\m_{t'}}}{|t-t'|^\b}\le R_0,\, 
	 \sup_{t\in  \R^+ }\int_{\R^d}|x|^k\m_t(dx)\le R_0,\,k=1,2\right\rbrace
\end{align*} 
	into itself, for an appropriate constant $R_0$ to be chosen later.  Note that the set $\cC$ is non-empty, since  it contains any  constant in time measure-valued map. Moreover, it is convex   since a  convex combination of maps in $\cC$   preserves the H\"older condition and boundedness of the moments.
	Lastly, the compactness is a consequence of Ascoli-Arzela's Theorem in space of measures (see \cite[Lemma 7.1.5]{ags} and \cite[pag.148]{cpt}). Indeed, given a  sequence $\{\m^n\}_{n\in \N}\subset \cC$, then for a given $T>0$, by Ascoli-Arzela's Theorem we can find a converging subsequence  in $[0,T]$. Iterating the same argument in $[kT,(k+1)T]$, $k\in\N$, we can find a subsequence $\m^{n_k}$ converging to $\m\in\cC$. 	\\
	Given $\nu\in \cC$, we set $\mu=\Psi_n(\nu)$. Since $m^{\n}_r= \Phi^{\nu}_r\#\m_0(dx)$, we have	
\begin{equation*}
	\m_t(\R^d)=\int_{\R^d}\m_t(dx)=\int_0^{T_n} \int_{\R^d}   m^{\n}_{r}(dx) h_\b(r,t)dr\le \int_0^\infty \int_{\R^d}  \m_0(dx) \, h_\b(r,t)dr=\m_0(\R^d).
\end{equation*}
To estimate the first and second order moments of $\m_t$, observe that
since $v[\n](x)\equiv 0$ in $\R^d\setminus B(0,2R)$, then 
$\vv^\n(x,s)\equiv 0$ for all $(x,s)\in(\R^d\setminus B(0,2R))\times \R^+$. Therefore, if $x\in B(0,R)$, it follows that $\Phi^{\nu}_r(x,0)\in B(0,2R)$ for all  $r\in \R^+$. Hence, for $k=1,2$, we get
\begin{equation}\label{moment_R}
\begin{split}
	&\int_{\R^d}|x|^k\m_t(dx)=\int_0^{T_n} \int_{\R^d}|x|^k  m^{\n}_{r}(dx) h_\b(r,t)dr\\
	&\le\int_0^\infty \int_{\R^d}   |\Phi^{\nu}_r(x,0)|^k \m_0(dx) \, h_\b(r,t)dr	\le (2R)^k\m_0(\R^d).
	\end{split}
\end{equation}

We estimate $\wass{\m_t}{\m_{t'}}$ for $0\le t'<t$. Arguing as in the estimates \eqref{m2_l}-\eqref{m3_I2}, for $f\in \Lip$ such that $\|f\|_{BL}\le 1$ we have
\begin{equation}\label{m3}
\begin{split}
&\dual{\m_t-\m_{t'}}{f}=\int_0^{T_n}  \left( \dual{m^{\nu}_{r}}{f}h_\b(r,t)-\dual{m^{\nu}_{r}}{f}h_\b(r,t')\right)dr \\
&=\int_0^{T_n}  \dual{m^{\nu}_{r}}{f}\left(h_\b(r,t)-h_\b(r,t')\right)dr
\le V_0\m_0(\R^d)C(\b,1) |t-t'|^\b.
\end{split}
\end{equation}
 By \eqref{moment_R} and \eqref{m3} (which are uniform in $n$), it follows that the map $\Psi_n$ is well defined. Moreover if $R_0$
 in the definition of $\cC$ is greater than the constants appearing in the estimates \eqref{moment_R}  and  \eqref{m3}, then $\Psi_n$ maps $\cC$ into itself.\\
  We now prove that $\Psi_n$ is continuous. We first estimate the distance between the flows corresponding to different measures. 
  Given $\m,\n\in C(\R^+,\cP)$, by  \textbf{(H2)}   we have
  \begin{align*}
  &|\vv^\m(x,s)-\vv^\n(x,s)|=\left|\int_0^\infty ( v[\m_r] -v[\n_r])g_\b(r,s)dr\right|\\ 
  &=\left|\E[v[\m_{D_s}](x)-v[\n_{D_s}](x)]\right|\le L \E[\wass{\m_{D_s}}{\n_{D_s}}].
  \end{align*}
  Hence, 
  \begin{align*}
  &\left|\Phi^\m_s(x,0)- \Phi^\n_s(x,0)\right|	=
  \int_0^s|\vv^\m(\Phi^\m_r(x,0),r)-\vv^\n(\Phi^\n_r(x,0),r)|dr\\
  &\le\int_0^s|\vv^\m(\Phi^\m_r(x,0),r)-\vv^\m(\Phi^\n_r(x,0),r)|dr+\int_0^s|\vv^\m(\Phi^\n_r(x,0),r)-\vv^\n(\Phi^\n_r(x,0),r)|dr
  \\
  &\le L\int_0^s |\Phi^\m_r(x,0)  - \Phi^\n_r(x,0) |dr
  +L \int_0^s  \E[\wass{\m_{D_r}}{\n_{D_r}}] dr.
  \end{align*}
  Therefore, by Gronwall's inequality, we have the estimate
  \begin{equation}\label{cont_1}
  \left|\Phi^\m_s(x,0)- \Phi^\n_s(x,0)\right| \le L e^{Ls}\int_0^s  \E[\wass{\m_{D_r}}{\n_{D_r}}] dr
  \le L se^{Ls}\E[\sup_{r\in[0,D_s]}  \wass{\m_{r}}{\n_{r}}].
  \end{equation}
Let $\{\nu^k\}_{k \in \N} \subset \mathcal{C}$ be a sequence converging to $\nu \in \mathcal{C}$ and set
  $\m^k=\Psi(\nu^k)$, $\m=\Psi(\nu)$. Denoted with $\Phi^{\nu^k}_r, \Phi^\nu_r$ the corresponding flows, we have for $f \in BL(\R^d)$
 \begin{align*}
     \langle \mu^k_t - \mu_t, f \rangle &= \int_0^{T_n}\langle m^{\nu^k}_s - m^{\nu}_s, f \rangle h_\beta(s, t)ds\\
     &= \int_0^{T_n} \int_{\R^d} \left( f(\Phi_s^{\nu^k}(x,0)) - f(\Phi_s^{\nu}(x,0)) \right)\mu_0(dx)h_\beta(s, t)ds\\
     &\leq \mu_0(\R^d) \int_0^{+\infty} \E\left[\sup_{r \in [0,D_s]} d_{BL}( \nu^k_r, \nu_r)\right]Lse^{Ls}\chi_{\{s < T_n\}}h_\beta(s,t)ds\\
     &\leq \mu_0(\R^d) \sup_{r\in[0,t]}\left(d_{BL}(\nu^k_r, \nu_r)\right)  \mathbb{E}[LE_t e^{LE_t} \chi_{\{E_t < T_n\}}] \\
     &\le \mu_0(\R^d)T_nL e^{LT_n} \sup_{r \in [0,t]}d_{BL}(\nu^k_r, \nu_r).
 \end{align*}
  Then, the convergence of $\mu^k$ to $\mu$, i.e. the continuity of the map $\Psi_n$.
  \\
  Applying  the Schauder's Fixed Point Theorem to $\Psi_n$, we have that, for every $n \in \N$, there exists   $\mu^n \in \mathcal{C}$ such that $\mu^n = \Psi_n(\mu^n)$, i.e.
 \begin{equation}\label{fix_n}
 \mu_t^n(dx) = \int_0^{T_n} m^{\mu_n}_s(dx)h_\beta(s,t)ds, \quad \forall t \in \R^+.
 \end{equation}
 Due to the the uniform in time estimates on first and second momentum and the uniform H\"olderianity, by a diagonal argument, we have the existence of a subsequence $\{\mu^{n_k}\}_{k \in \N}$ and $\mu \in \mathcal{C}$ such that $\mu^{n_k}_t \to \mu_t$   for $t \in \R^+.$ 
 We     prove that for any $t\in\R^+$, $\mu$   satisfies the integral formula   \eqref{integral_nl}.
 Denote with $\{\mu^n\}_{n \in \N}$ the (sub-)sequence converging to $\mu$, for any $f \in BL(\R^d)$ with $\|f\|_{BL}\leq 1$ we have
 \begin{align*}
      &\int_0^{+\infty}\langle m^\mu_s, f \rangle h_\beta(s,t)ds - \int_0^{T_n} \langle m^{\mu^n}_s, f \rangle h_\beta(s,t)ds =\\
      &= \int_{T_n}^{+\infty}\langle m^\mu_s, f \rangle h_\beta(s,t)ds + \int_0^{T_n} \langle m^\mu_s - m^{\mu^n}_s, f \rangle h_\beta(s,t)ds.
 \end{align*}
  The first term on the right hand side converges to 0 for $n \to +\infty$, since
 \begin{align}
     \int_{T_n}^{+\infty}\langle m^\mu_s, f \rangle h_\beta(s,t)ds \leq \mu_0(\R^d) \int_{T_n}^{+\infty} h_\beta(s, t)ds \to 0.
   \end{align}
 For the second term,  by \eqref{cont_1} we have 
 \begin{align*}
      \int_0^{T_n} \langle m^\mu_s - m^{\mu^n}_s, f \rangle h_\beta(s,t)ds&\leq \int_0^{T_n} \int_{\R^d}\left( f(\Phi^\mu_s(x, 0)) - f(\Phi^{\mu^n}_s(x,0)) \right)\mu_0(dx)h_\beta(s,t) ds\\
      &\leq\mu_0(\R^d)\int_0^{T_n}\E\left[\sup_{r \in [0, D_s]}d_{BL}(\mu^n_r, \mu_r)\right]Lse^{Ls}h_\beta(s, t)ds\\
      &\leq \mu_0(\R^d)\E[LE_t e^{LE_t}]\sup_{[0,t]}d_{BL}(\mu^n_s, \mu_s).
 \end{align*}
  Note that  by Lemma \ref{lemm1},  $\E[LE_te^{LE_t}]
  <\infty$  for any  $t\geq 0$.
 Hence, for $n \to +\infty$ and $t \in \R^+$, we have 
 $$ \int_0^{T_n} m^{\mu^n}_s (dx) h_\beta(s, t) ds \to \int_0^{+\infty} m^{\mu}_s (dx) h_\beta(s, t) ds$$
 and passing to the limit for $n\to \infty$ in \eqref{fix_n}, we get  \eqref{integral_nl}. \\
 As in Proposition \ref{existence_linear},
 it is possible to show that formula \eqref{integral_nl} defines a weak solution of the problem. Hence the existence of a solution 
 to \eqref{caputo_nonlinear} follows.
\end{proof} 
\begin{proof}[Proof of Theorem \ref{existence_nl}]
	Given $R>0$, we consider a sequence of initial data given by $\m_0^R=\chi_{B(0,R)}(x)\cdot\m_0$, where $\chi_{B(0,R)}$ is the characteristic function of the set $B(0,R)$, and a sequence of velocity fields $v^R[\m](x)=v[\m](x)\cdot\s_R(x)$
	where $\s_R:\R^d\to\R$ is a smooth, non negative function such that $\s_R(x)=1$ for $x\in B(0,2R-1)$, $\s_R(x)=0$ for $x\in\R^d\setminus B(0,  2R)$ and $|D\s_R|\le 1$. By Lemma  \ref{existence_compact}, for any $R>1$, there exists   a solution $\m^R$ to the Cauchy problem \eqref{caputo_nonlinear} given by the formula
	\begin{equation}\label{integral_R}
	\m^R_t(dx)= \int_0^\infty m^R_s(dx)h_\b(s,t)ds=\int_0^\infty \Phi^R_s\#\m^R_0(dx) \, h_\b(s,t)ds
	\end{equation}
where $m^R$ is the solution of \eqref{caputo_nonlinear} corresponding to the  velocity $\vv^R(x,s)=\int_0^\infty v^R[\m^R_r](x)g_\b(r,s)dr$ and $\Phi^R$ the associated flow.\\
 For any $T>0$, we consider the restriction of the sequence $\{\m^R\}$ to the interval $[0,T]$. We estimate the first and second order moment of $\m^R_t$ for $t\in [0,T]$, uniformly with respect to   $R$. By \eqref{flow_nl}, we have
\begin{align*}
&\int_{\R^d}|x|\m^R_t(dx)=\int_0^\infty \int_{\R^d}   |\Phi^R_s(x,0)| \m^R_0(dx) \,  h_\b(s,t)ds \\
&\le \int_0^\infty  \int_{\R^d}|x| \m^R_0(dx)     h_\b(s,t)ds+ \mu_0(\R^d)V_0\, \int_0^\infty s   h_\b(s,t)ds\le \int_{\R^d}|x| \m_0(dx)  +  \mu_0(\R^d)V_0C(\b,1)t^\b,
\end{align*}
and 
\begin{align*}
&\int_{\R^d}|x|^2\m^R_t(dx)=\int_0^\infty  \int_{\R^d}  |\Phi^R_s(x,0)|^2 \m^R_0(dx) \, h_\b(s,t)ds\le \int_0^\infty  \left( \int_{\R^d}|x|^2 \m^R_0(dx)\right)     h_\b(s,t)ds\\
&+  V_0^2 \int_0^\infty \left(  \int_{\R^d}\m^R_0(dx)\right)  s^2  h_\b(s,t)ds +2 V_0   \int_0^\infty  \left( \int_{\R^d}  |x|   \m^R_0(dx)\right)  \,s h_\b(s,t)ds \\
&\le \int_{\R^d}|x|^2 \m_0(dx)+ V_0^2\m_0(\R^d) C(\b,2)t^{2\b}  +2V_0C(\b,1)t^\b \int_{\R^d}|x| \m_0(dx) .
\end{align*}
The previous estimates imply that the sequence  $\mu^R_t$ is relatively compact in $\cP$ for any $t\in [0,T]$.
Moreover  estimate    \eqref{m3}, which is independent of $R$, implies that the sequence $\mr$  is also equi-continuous with respect to $t$. By Ascoli-Arzela's Theorem (see \cite[Section 3.3 and Lemma 7.1.5]{ags}), we conclude that  there exists a measure $\m\in C([0,T],\cP)$ such that $\sup_{[0,T]}\wass{\m^R_t}{\m_t}$
tends to 0 for $R\to \infty$, up to a subsequence. By a diagonal argument, we can extend   $\m\in C(\R^+,\cP)$.\par
To prove that $\m$ is a solution of \eqref{caputo_nonlinear}, we show that it satisfies the integral formula \eqref{integral_nl}. 
Let $m^\m\in C(\R^+,\cP)$ be the solution of the Cauchy problem \eqref{dual_nonlinear} associated to the measure $\m$,
i.e. with velocity field $\vv^\m(x,s)=\int_0^\infty v[\m_r](x)g_\b(r,s)dr$, and initial datum $\m_0$. 
Then,  for $f\in\Lip$ such that$\|f\|_{BL}\le 1$, we have
\begin{equation}\label{convergence}
\begin{split}
\dual{m^R_s-m^\m_s}{f}=\int_{\R^d}f(\Phi^R_s(x,0))\m^R_0(dx)-\int_{\R^d}f(\Phi^\m_s(x,0))\m_0(dx)=\\
\int_{\R^d}\left(f(\Phi^R_s(x,0))-f(\Phi^\m_s(x,0))\right)\m^R_0 (dx)+   \int_{\R^d}f(\Phi^\m_s(x,0))(\m^R_0(dx)- \m_0(dx))\\
 \le \int_{\R^d}\left|\Phi^\m_s(x,0)- \Phi^R_s(x,0)\right|\m^R_0 (dx)+
 \m_0(\R^d\setminus B(0,R)).
\end{split}
\end{equation}
 To estimate the first term on the right hand side, we preliminary  observe
 by \textbf{(H2)}    
\begin{align*}
&|\vv^\m(x,s)-\vv^R(x,s)|=\left|\int_0^\infty ( v[\m_r](x) -v[\m^R_r](x)\s^R(x))g_\b(r,s)dr\right|\\ 
&\le\int_0^\infty | v[\m_r](x) -v[\m^R_r](x)|g_\b(r,s)dr+\int_0^\infty | v[\m^R_r](x) (1-\s^R(x))|g_\b(r,s)dr\\
&\le L \E[\wass{\m_{D_s}}{\m^R_{D_s}}]+V_0\chi_{\R^d\setminus B(0,2R-1)}(x).
\end{align*}
Hence 
\begin{align*}
&\left|\Phi^\m_s(x,0)- \Phi^R_s(x,0)\right|	=
\int_0^s|\vv^\m(\Phi^\m_r(x,0),r)-\vv^R(\Phi^R_r(x,0),r)|dr\\
&\le \int_0^s|\vv^\m(\Phi^\m_r(x,0),r)-\vv^\m(\Phi^R_r(x,0),r)|dr
+\int_0^s|\vv^\m(\Phi^R_r(x,0),r)-\vv^R(\Phi^R_r(x,0),r)|dr\\
&\le L\int_0^s |\Phi^\m_r(x,0)  - \Phi^R_r(x,0) |dr
+\int_0^s\left( L \E[\wass{\m_{D_r}}{\m^R_{D_r}}]+V_0\chi_{\R^d\setminus B(0,2R-1)}(\Phi^R_r(x,0))\right)dr
\end{align*}
and therefore, by Gronwall's inequality, we get
\begin{equation}\label{conv_1}
\begin{split}
\left|\Phi^\m_s(x,0)- \Phi^R_s(x,0)\right|&\le  e^{Ls}\int_0^s\left(L \E[\wass{\m_{D_r}}{\m^R_{D_r}}]+V_0\chi_{\R^d\setminus B(0,2R-1)}(\Phi^R_r(x,0))\right)dr.
\end{split}
  \end{equation}
Replacing  \eqref{conv_1} in \eqref{convergence}, we get
\begin{align*} 
\dual{m^R_s-m^\m_s}{f}&\le e^{Ls}\int_0^s \left(L  \E[\wass{\m_{D_r}}{\m^R_{D_r}}]+V_0\int_{\R^d}\chi_{\R^d\setminus B(0,2R-1)}(\Phi^R_r(x,0))\m_0^R(dx)\right) dr\\
&  +\m_0(\R^d\setminus B(0,R)).
\end{align*}
and therefore, as in \eqref{cont_1}, we get
\begin{align} 
&\int_0^\infty\dual{m^R_s-m^\m_s}{f}h_\b(s,t)ds   \le \int_0^\infty
 \sup_{r\in[0,s]}  \E[\wass{\m_{D_r}}{\m^R_{D_r}}] Ls e^{Ls} h_\b(s,t)ds\nonumber\\
&+ \int_0^\infty e^{Ls}\int_0^s\int_{\R^d} V_0  \chi_{\R^d\setminus B(0,2R-1)}(\Phi^R_r(x,0))\m_0^R(dx) h_\b(s,t)dr ds+   \m_0(\R^d\setminus B(0,R)) \label{est_final}\\
& \le \sup_{r\in[0,t]} \wass{\m_{r}}{\m^R_{r}} \E[L E_te^{LE_t}]
+\E\left[e^{LE_t}\int_0^{E_t}\int_{\R^d}V_0  \chi_{\R^d\setminus B(0,2R-1)}(\Phi^R_r(x,0))\m_0^R(dx)dr\right]\nonumber\\
&+ \m_0(\R^d\setminus B(0,R)).\nonumber
\end{align}
Observe that if $x\in B(0,R)$, then
\[ \E[\sup_{[0,E_t]}|\Phi^R_r(x,0)| ]\le |x|+V_0C(\b,1)t^\b .\]
Since    inequality \eqref{est_final} holds for any $f$ and all the terms on the right hand side tend to $0$ for $R\to \infty$,  we get the convergence of $\int_0^\infty m^R_s(dx)h_\b(s,t)ds$ to
$\int_0^\infty m^\m_s(dx)h_\b(s,t)ds$ for $R\to \infty$.
Passing to the limit for $R\to \infty$ in \eqref{integral_R}, we get that  $\m$ satisfies \eqref{integral_nl} and therefore it is a weak solution to \eqref{caputo_nonlinear}.\par
We finally prove  the uniqueness of the solution to \eqref{caputo_nonlinear}. Let $\mu^1, \mu^2$ be two solutions to \eqref{caputo_nonlinear} with initial conditions $\mu^1_0, \mu^2_0 \in \cM^+(\R^d)$. 
 For $f \in BL(\R^d)$ with $\|f\|_{BL}\leq 1$, we have
\begin{align*}
    \langle \mu^1_t - \mu^2_t,f \rangle &= \int_0^{+\infty} \langle m^1_s - m^2_s, f\rangle h_\b(s,t)ds \\
    &= \int_0^{+\infty} \left(\langle \mu^1_0, f(\Phi^{\m^1}_s) - f(\Phi^{\m^2}_s)\rangle + \langle \mu^1_0 - \mu^2_0, f(\Phi^{\m^2}_s) \rangle\right)h_\b(s,t)ds\\
    &\leq \int_0^{+\infty} \left(\langle \mu^1_0 \left|  \Phi^{\m^1}_s - \Phi^{\m^2}_s\right|\rangle + \langle \mu^1_0 - \mu^2_0, f(\Phi^2_s) \rangle\right)h_\b(s,t)ds.
\end{align*}
By the first inequality in \eqref{cont_1}, i.e. $\left|\Phi^\m_s(x,0)- \Phi^\n_s(x,0)\right| \le L e^{Ls}\int_0^s  \E[\wass{\m_{D_r}}{\n_{D_r}}] dr$, and since $L^{-1}f(\Phi^{\m^2}_s(x,0))e^{-Ls}$ is $1$-Lipschitz, it follows that
\begin{align*}
\langle \mu^1_t - \mu^2_t,f \rangle &\leq \mu^1_0(\R^d) \int_0^{+\infty} L e^{Ls} \left(\int_0^s \L(r)dr\right)h_\b(s,t)ds + L d_{BL}(\mu^1_0, \mu^2_0)\int_0^{+\infty}e^{Ls}h_\b(s,t)ds\\
&= \mu^1_0(\R^d) \int_0^{+\infty} L e^{Ls} \left(\int_0^s \L(r)dr\right)h_\b(s,t)ds + L d_{BL}(\mu^1_0, \mu^2_0)\E[e^{LE_t}],
\end{align*}
where $\Lambda(r) = \mathbb{E}[d_{BL}(\mu^1_{D_r}, \mu^2_{D_r})]$. For the arbitrariness of $f$ and recalling that $h_\b(\cdot,t)$ is the PDF of the process $E_t$, we get
\begin{equation}\label{es110}
  d_{BL}(\mu^1_{t}, \mu^2_{t})  \le  \mu^1_0(\R^d)\E\left[  L e^{LE_t}  \int_0^{E_t} \L(r)dr \right] + L d_{BL}(\mu^1_0, \mu^2_0)\E[e^{LE_t}].
\end{equation}
Observe that, by the conservation of the initial mass, we have $\Lambda(r) \leq \mathbb{E}[\|\mu^1_{D_r}\|^*_{BL} + \|\mu^2_{D_r}\|^*_{BL}] = \|\mu^1_0\|^*_{BL} + \|\mu^2_0\|^*_{BL}$  and therefore $\Lambda(r)$ is finite for any $r \geq 0$. Moreover,  replacing $t=D_r$ in \eqref{es110} and taking the expectation, we get 
\begin{align*}
\L(r) &\leq \mu^1_0(\R^d)   \mathbb{E}\left[L e^{L E_{D_r}} \int_{0}^{E_{D_r}} \L (z) dz \right] + L d_{BL}(\mu^1_0, \mu^2_0)\mathbb{E}[e^{L E_{D_r}}]\\
&=  \mu^1_0(\R^d) L e^{L r} \int_{0}^{r} \L (z) dz  + d_{BL}(\mu^1_0, \mu^2_0)L e^{L r},
\end{align*}
since $E_t$ is the inverse   process of $D_t$. Then, by Gronwall's inequality, we have 
\begin{equation}\label{es111}
\L(r) \leq Le^{Lr} d_{BL}(\mu^1_0, \mu^2_0) e^{\mu^1_0(\R^d)(e^{Lr} - 1)}.
\end{equation}
Hence, if we put $\mu^1_0=\mu^2_0$ in the previous inequality, it follows that $\Lambda(r)=0$, i.e. 
\[
\int_0^{+\infty} d_{BL}(\mu^1_s,\mu^2_s)g_{\b}(s,r)ds=0.
\]
Since $g_{\b}(s,r) > 0$ for $(s,r) \in \R^+\times \R^+$ and $\m^i \in  C( \R^+,\cP)$, $i=1,2$, it follows that $d_{BL}(\mu^1_s,\mu^2_s) = 0$ for  $s \in \R^+$, i.e. the uniqueness in measure theoretic sense of the solution to \eqref{caputo_nonlinear}.
\end{proof}

\begin{remark}\label{source}
Consider a nonlinear transport equation with source term
\begin{equation}\label{transport_source}
\left\{\begin{array}{ll}
\fdz \m+\diver(v[\m_t]\,\m)=\g_t\quad & (x,t)\in \R^d\times (0,T), \\[8pt]
\m_{t=0}=\m_0  & x\in\R^d,
\end{array}
\right.
\end{equation}	
where $\g\in C([0,T],\cM^+(\R^d))$ with bounded first and second order moments. Then the solution of \eqref{transport_source} is
given by 
\[ \m_t(dx)= \E[m_{E_t}(dx)]=\int_0^\infty m^\m_s(dx)h_\b(s,t)ds \]
where  $m^\m\in C(\R^+,\cP)$	is the  solution of the linear transport problem
\[ \left\{\begin{array}{ll}
\pd_s m+\diver(\vv^\m(x,s)\,m)=\mM_s\quad & (x,s)\in \R^d\times\R^+, \\[8pt]
m_{s=0}=\m_0  & x\in\R^d,
\end{array}
\right. \]
with $\vv^\m$  defined as in \eqref{dual_vel_nonlinear} and the source term   given by
\[ \mM_s(dx)=\int_0^\infty \g_r(dx)\, g_\b(r,s)ds. \]
In terms of the push-forward by means of the flow $\Phi^\m$ associated to $\vv^\m$, the solution of \eqref{transport_source}
can be also written as
\[ 	\m_t(dx)= \int_0^\infty \left( \Phi^{\m}_s\#\m_0(dx)+\int_0^s\Phi^{\m}_{s-r}\#\mM_r(dx)\,dr\right) \, h_\b(s,t)ds. \]
\end{remark}


\end{document}